\documentclass[10pt]{amsart}

\usepackage{amssymb,latexsym,amscd,amsmath,amsthm}
\usepackage{epsfig}
\usepackage{psfrag}
\usepackage{graphicx}
\usepackage{color}
\usepackage{geometry}
\usepackage{multirow}
\usepackage{hyperref}
\usepackage{comment}
\usepackage{palatino}

\usepackage[all]{xy}
\xyoption{all}
\xyoption{arc}
\xyoption{poly}

\usepackage{tikz}
\usetikzlibrary{backgrounds,calc}

\usepackage[hang,small,bf]{caption}
\newtheorem{thm}{Theorem}[section]
\newtheorem{lem}[thm]{Lemma}

\newtheorem{cor}[thm]{Corollary}
\newtheorem{prop}[thm]{Proposition}

\theoremstyle{plain}
\newtheorem{main}{Theorem}
\newtheorem{corm}[main]{Corollary}

\theoremstyle{definition}

\newtheorem{rem}[thm]{Remark}
\newtheorem*{ack}{Acknowledgments}

\DeclareMathOperator{\symrank}{symrank}
\DeclareMathOperator{\Iso}{Iso}

\newcommand{\R}{\mathbf{R}}
\newcommand{\Z}{{\mathbf{Z}}}

\newcommand{\C}{{\mathbf{C}}}
\newcommand{\cp}{\mathbf{CP}}

\newcommand{\sph}{\mathbf{S}}

\newcommand{\G}{\operatorname{G}}
\newcommand{\M}{\operatorname{M}}

\newcommand{\HH}{\operatorname{H}}
\newcommand{\SU}{\operatorname{SU}}
\newcommand{\SO}{\operatorname{SO}}

\newcommand{\U}{\operatorname{U}}

\newcommand{\Ker}{\operatorname{Ker}}

\renewcommand{\lim}[1]{\mathop{\underset{#1} {\underset \longleftarrow
{\text{\rm lim}}}}}

\newcommand{\T}{\operatorname{T}}
\newcommand{\B}{\operatorname{B}}
\newcommand{\p}{\operatorname{P}}
\newcommand{\fix}{\operatorname{Fix}}

\newcommand{\bq}{/ \hspace{-.15cm} /}
\newcommand{\lra}{\longrightarrow}

\newcommand{\vep}{\varepsilon}
\def\x{\times}
\def\<{\langle}
\def\>{\rangle}
\def\met{\<\, ,\, \>}

\newcommand{\SUU}[1]{\operatorname{S}(\U(#1) \x \U(#1))}

\def\bsm{\begin{smallmatrix}}
\def\esm{\end{smallmatrix}}
\def\bpm{\begin{pmatrix}}
\def\epm{\end{pmatrix}}
\def\bbm{\begin{bmatrix}}
\def\ebm{\end{bmatrix}}
\def\beq{\begin{equation}}
\def\eeq{\end{equation}}

\renewcommand{\geq}{\geqslant}
\renewcommand{\leq}{\leqslant}

\numberwithin{equation}{section}

\newcommand{\spacing}[1]{\renewcommand{\baselinestretch}{#1}\large\normalsize}
\spacing{1.14}

\begin{document}




\title[Cohomogeneity-two torus actions in low dimensions]{Cohomogeneity-two torus actions on non-negatively curved manifolds of low dimension}

\author[F.\ Galaz-Garcia and M.\ Kerin]{Fernando Galaz-Garcia and Martin Kerin$^\ast$}

\thanks{$^\ast$ This research was carried out as part of SFB 878: Groups, Geometry \& Actions, at the University of M\"unster.}

\address{Mathematisches Institut, Einsteinstr. 62, 48149 M\"unster, Germany}
\email{f.galaz-garcia@uni-muenster.de}
\email{m.kerin@math.uni-muenster.de}




\begin{abstract}
Let $\M^n$, $n \in \{4,5,6\}$, be a compact, simply connected $n$-manifold which admits some Riemannian metric with non-negative curvature and an isometry group of maximal possible rank.  Then any smooth, effective action on $\M^n$ by a torus $\T^{n-2}$ is equivariantly diffeomorphic to an isometric action on a normal biquotient.  Furthermore, it follows that any effective, isometric circle action on a compact, simply connected, non-negatively curved four-dimensional manifold 
 is equivariantly diffeomorphic to an effective, isometric action on a normal biquotient.
\end{abstract}

\date{\today}

\maketitle




\normalsize
\thispagestyle{empty}




The classification of (compact) Riemannian manifolds $(\M^n, g)$ with positive or non-negative (sectional) curvature is a notoriously difficult problem.
One of the few successes in this quest occurs when one considers such manifolds equipped with an effective action by a suitably large group $\G$ of isometries.  The ambiguity of the term ``suitably large'' allows various classification results to be achieved (see, for example, \cite{BB, GS94, GS, GWZ08, Ve, Wa, Wi1, Wi2} and the surveys \cite{G, Wi3}).

There are, in fact, two parts to the classification program.  First is the \emph{topological classification}, the goal of which is to determine, up to diffeomorphism, all possible positively or non-negatively curved manifolds on which $\G$ can act. The second part is the \emph{equivariant classification}, where the goal is to determine all possible actions of $\G$ on a given positively or non-negatively curved manifold up to equivariant diffeomorphism.

This approach to the classification problem is inspired by the work of Hsiang and Kleiner \cite{HK}.  They showed that a simply connected, four-dimensional manifold $(\M^4, g)$ with positive curvature admitting an effective, isometric circle action must be homeomorphic to either $\sph^4$ or to $\cp^2$.  If $(\M^4, g)$ is assumed to have only non-negative curvature then, by Kleiner \cite{Kl} and Searle and Yang \cite{SY}, $\M^4$ must be homeomorphic to one of $\sph^4$, $\cp^2$, $\cp^2 \# \pm \cp^2$ or $\sph^2 \x \sph^2$.  In both situations, homeomorphism is improved to diffeomorphism by appealing to work of Fintushel \cite{F1,F2}, Pao \cite{Pa} and Perelman's proof of the Poincar\'e conjecture \cite{P1,P2}.

The existence of an effective, isometric circle action is equivalent to the rank of the isometry group $\Iso(\M^n, g)$ being positive.  The success in dimension four suggests that it may be beneficial to consider the topological classification when ``largeness'' of our group $\G$ of isometries refers to the \emph{symmetry rank} of $(\M^n,g)$, defined as the rank of $\Iso(\M^n, g)$ and denoted $\symrank(\M^n, g)$.  Indeed, Grove and Searle \cite{GS94} showed that if $(\M^n, g)$ is positively curved, then $\symrank(\M^n, g) \leq \lfloor \tfrac{n+1}{2} \rfloor$. Moreover, if the symmetry rank is maximal, namely $\symrank(\M^n, g) = \lfloor \tfrac{n+1}{2} \rfloor$, then $\M^n$ is diffeomorphic to a sphere, a real or complex projective space, or a lens space.

In the non-negative curvature setting, when $(\M^n, g)$ is simply connected and $n \leq 9$, it is known that $\symrank(\M^n, g) \leq \lfloor \tfrac{2n}{3} \rfloor$.  If equality holds, then $\M^n$ is diffeomorphic to $\sph^4$, $\cp^2$, $\cp^2 \# \pm \cp^2$ or $\sph^2 \x \sph^2$ for $n=4$; to $\sph^3 \x \sph^2$, $\sph^3\tilde{\x}\sph^2$ (the non-trivial $\sph^3$-bundle over $\sph^2$) or $\sph^5$ for $n=5$; and to $\sph^3 \x \sph^3$ for $n=6$ (cf. \cite{GGS10}).

This article is concerned with the equivariant classification in low dimensions.  Before the statement of the main result, it is necessary to recall that a \emph{biquotient} is a quotient of a Lie group $\G$ by the two-sided, free action of a subgroup $\U \subset \G \x \G$.  If $\G$ is equipped with a bi-invariant metric, then the action of $\U$ is by isometries and the quotient $G \bq \U$ equipped with the induced metric (of non-negative curvature) is called a \emph{normal biquotient}.

\begin{main}
\label{Thm A}
Let $\M^n$, $n \in \{4,5,6\}$, be an $n$-manifold which admits a Riemannian metric with non-negative curvature and maximal symmetry rank.  Then every smooth, effective action on $\M^n$ by a torus $\T^{n-2}$ is equivariantly diffeomorphic to an effective, isometric action on a normal biquotient.
\end{main}

In light of the topological classification discussed earlier, we make the following remarks.  Any smooth, effective $\T^4$ action on $\sph^3\x\sph^3$ is equivariantly diffeomorphic to the standard action of $\T^4$ on the Lie group $\sph^3\x\sph^3$ equipped with a bi-invariant metric (cf. \cite{Oh82}).  From \cite{Oh83} it follows that any smooth, effective $\T^3$ action on $\sph^5$ is equivariantly diffeomorphic to a linear action (on a normal homogeneous space).  Similarly, by the classification result of Orlik and Raymond \cite{ORa70}, any smooth, effective $\T^2$ action on $\sph^4$ or $\cp^2$ is equivariantly diffeomorphic to a linear action (on a normal homogeneous space).

In order to establish Theorem \ref{Thm A}, it therefore remains only to consider the actions of maximal rank tori on the manifolds $\cp^2 \# \pm \cp^2$, $\sph^2 \x \sph^2$, $\sph^3\x \sph^2$ and $\sph^3\tilde{\x}\sph^2$.  Each of these manifolds can be described as a normal biquotient of $\sph^3\x \sph^3$ by the free action of either a two-torus or circle (cf. \cite{Ch}, \cite{To}, \cite{Pav04}, \cite{DV}).  The standard effective, isometric action of $\T^4$ on the Lie group $\sph^3\x\sph^3$ induces a maximal rank, effective, isometric torus action on each of these biquotients.  An examination of all possible induced actions on these biquotients then shows that every equivariant diffeomorphism type appearing in the classifications of Orlik and Raymond \cite{ORa70} and Oh \cite{Oh83} is achieved.

Grove and Wilking \cite{GW} have recently shown that an effective, isometric circle action on $\sph^4$ or $\cp^2$ is equivariantly diffeomorphic to a linear action.  Furthermore, any effective, isometric circle action on a simply connected $4$-manifold with non-negative curvature extends to a smooth, effective $\T^2$ action (cf. \cite{GG}, \cite{GW}).  This fact, in combination with Theorem \ref{Thm A}, can be applied to $\cp^2 \# \pm \cp^2$ and $\sph^2 \x \sph^2$ to draw the following conclusion.

\begin{corm}
\label{Cor B}
An effective, isometric circle action on a non-negatively curved, simply connected $4$-manifold is equivariantly diffeomorphic to an effective, isometric action on a normal biquotient.
\end{corm}

It is worthwhile to remark that there exist smooth, effective (non-isometric) circle actions on $\sph^4$ which do not extend to a smooth $\T^2$ action (cf. \cite{Pa}). 

The paper is divided as follows. In Section \ref{S:Prelim} definitions and notation necessary for the rest of the paper are gathered.  In Section \ref{S:cohom2tori} some general facts about smooth, cohomogeneity-two torus actions are recalled, while in Sections \ref{S:dim4} and \ref{S:dim5} all possible orbit spaces of smoooth, cohomogeneity-two torus actions on the manifolds $\cp^2 \# \pm \cp^2$, $\sph^2 \x \sph^2$, $\sph^3 \x \sph^2$ and $\sph^3 \tilde{ \x} \sph^2$ are detailed.  In Section \ref{S:Biquotients} descriptions of these manifolds as biquotients are given and in Sections \ref{S:4dim} and \ref{S:5dim} it is shown that isometric, cohomogeneity-two actions on these biquotients recover all of the orbit spaces described in Sections \ref{S:dim4} and \ref{S:dim5}, thus proving Theorem \ref{Thm A}.  Finally, in Section \ref{S:P_bundles} we make some observations about principal circle bundles over the $4$-manifolds $\cp^2 \# \pm \cp^2$ and $\sph^2 \x \sph^2$.


\begin{ack} The first named author thanks B.\ Wilking and K.\ Grove for useful conversations.  The second named author wishes to thank J.\ DeVito for several interesting and helpful discussions.
\end{ack}


\section{Basic definitions and notation}
\label{S:Prelim}

Let $\G\x \M \rightarrow \M$,  $\ m \mapsto g \star m$, be a smooth action of a compact Lie group $\G$ on a smooth manifold $\M$.  We will denote the orbit space $\M/\G$ of the action  by $\M^*$. The \emph{cohomogeneity} of the action is the dimension of the orbit space $M^*$.  The orbit $\G \star p$ through a point $p \in \M$ is diffeomorphic to the quotient $\G/ \G_p$, where $\G_p := \{g \in \G \mid g \star p = p \}$ is the isotropy subgroup of $\G$ at $p$.  If $\G_p$ acts trivially on the normal space to the orbit at $p$, then $\G / \G_p$ is called a \emph{principal orbit}.  The set $P$ of principal orbits is open and dense in $M$.  In fact, the isotropy groups of principal orbits are all conjugate in $\G$, hence each principal orbit has the same dimension.

If $\G/\G_p$ has the same dimension as a principal orbit, but $\G_p$ acts non-trivially on the normal space at $p$, then $\G / \G_p$ is called an \emph{exceptional orbit}.  An orbit that is neither principal nor exceptional is called a \emph{singular orbit}.  The set of exceptional orbits will be denoted by $E$ and the set of singular orbits by $Q$.  When $\G_p = \G$, $p$ is called a \emph{fixed point} of the action.  The set of fixed points of the action will be denoted by $\fix(\M,\G)$. 

Given a subset $X \subset \M$, we will denote its projection under the orbit map $\pi: \M \rightarrow \M^*$ by $X^*$. Given a subset $X^*\subset \M^*$, we will let $X=\pi^{-1}(X^*)$ be its pre-image under $\pi$. Recall that the \emph{ineffective kernel} of the action  is $\Ker := \{g \in \G \ | \ g \star m = m, \ \forall \ m \in \M\}$.  The action is \emph{effective} if the ineffective kernel is trivial.  The group $\widetilde \G := \G/\Ker$ will always act effectively.  We say that the action of $\G$ is \emph{free} if $g \star m = m$ for some $m \in \M$ implies that $g = e$, the identity in $\G$.  In this case the orbit space $\M / \G$ is a manifold.  We may expand our definition of freeness to allow an ineffective kernel, namely $g \star m = m$ for some $m \in \M$ implies that $g \star m = m$ for all $m \in \M$.  In this case, the orbit space $\M/\G$ is again a manifold, diffeomorphic to $\M/ \widetilde \G$, where $\widetilde \G = \G/\Ker$ acts freely and effectively on $\M$.  If, in addition, there is a metric on $\M$ such that the action of $\G$ is by isometries, then there is an induced metric on $\M / \G$, given by the distance between orbits in $\M$.

Let $a_1, \dots, a_n \in \Z$ be relatively prime integers.  Define the circle subgroup of slope $\underline a = (a_1, \dots, a_n) \in \Z^n$ in $\T^n$ via $\G(\underline a) = \G(a_1, \dots, a_n) := \{z^{a_1}, \dots, z^{a_n}) \mid z \in \C, \ |z|^2 = 1 \}$.

By the \emph{determinant of $n$ circle subgroups} $\G(\underline a_{1}), \dots, \G(\underline a_{n})$ of $\T^n$ we mean the determinant of the $(n\x n)$-matrix
\[
\bpm
	a_{11} &  a_{12} & \dots & a_{1n} \\
	a_{21} &  a_{22} & \dots & a_{2n} \\
    \vdots & \vdots & & \vdots \\
	a_{31} &  a_{n2} & \dots & a_{nn}
	\epm
\]
where $\underline a_i = (a_{i1}, \dots, a_{in}) \in \Z^n$.

We recall that  two subgroups $\G(\underline a_1)$ and $\G(\underline a_2)$ of $\T^n$ have trivial intersection if and only if there exist $\G(\underline a_3), \dots, \G(\underline a_n) \subset \T^n$ such that the determinant of these $n$ circle subgroups is $\pm 1$.  A collection of $n$ circle subgroups $\G(\underline a_i)$, $i=1, \dots, n$, span $\T^n$ provided their determinant is non-zero, i.e. the vectors $\underline a_1, \dots, \underline a_n$ span $\R^n$.  They are generators of $\T^n$, that is, $\T^n=\G(\underline a_1) \x \dots \x \G(\underline a_n)$, if and only if their determinant is $\pm 1$.

Consider now a closed subgroup $\U \subset \G \x \G$ acting on a compact Lie group $\G$ via
$$
(u_1, u_2) \star g = u_1 g u_2^{-1}, \ g \in \G, (u_1, u_2) \in \U.
$$
This action is free if and only if, for non-trivial $(u_1, u_2) \in \U$, $u_1$ is never conjugate to $u_2$ in $\G$.  The resulting quotient manifold, denoted $\G \bq \U$, is called a \emph{biquotient}.  As discussed above, we may extend the definition of biquotient to allow actions which act freely up to an ineffective kernel  (cf. \cite{Es}).  In particular, if $\G$ is equipped with a bi-invariant metric $\met_0$, then $\U$ acts by isometries and $(\G, \met_0) \bq \U$ is called a \emph{normal biquotient}.


\section{Cohomogeneity-two torus actions}
\label{S:cohom2tori}

Let $\M^{n+2}$  be a smooth, compact, simply connected $(n+2)$-manifold on which a compact, connected Lie group $\G$ acts smoothly, effectively and with cohomogeneity two. It is well-known (see, for example, \cite[Chapter IV]{Br}) that in the presence of singular orbits, the orbit space $\M^*$ of the action is homeomorphic to a $2$-disk $D^2$ with boundary $Q^*$, the projection of  the singular orbits, while the interior points of $\M^*$ correspond to  principal orbits. When $\G=\T^n$, $n\geq 2$, the orbit space structure  was analyzed in \cite{ORa70, KMP} (see also \cite{Oh83}). In this case the only possible non-trivial isotropy groups are $\sph^1$ and $\T^2$. The boundary circle, $Q^*$,  is a union of arcs, and the interior of each arc corresponds to orbits with isotropy $\sph^1$, while the endpoints of each arc correspond to orbits with isotropy $\T^2$ (see Figure 1). Moreover, there must be at least $n$ orbits with isotropy $\T^2$.

\begin{center}
  \begin{tikzpicture}[scale=2]
    \path[coordinate] (0.5,0.866)  coordinate(A)
                (0,1) coordinate(B)
                (-0.5,0.866) coordinate(C)
                (-0.866,0.5) coordinate(D)
                (-1,0) coordinate(E)
                (-0.866,-0.5) coordinate(F)
                (-0.5,-0.866) coordinate(G)
                (0,-1) coordinate(H)
                (0.5,-0.866) coordinate(I);

    \draw (A) -- (B) -- (C) -- (D) -- (E) -- (F) -- (G) -- (H) -- (I);
    \draw[loosely dotted] (0.5,-0.866) arc (-60:60:1);

    \draw[color = black] (1.1, 0.8) node{$\M^*$};

    \draw[color = black] (2, -0.2) node{trivial isotropy};
    \draw[stealth-] (0,0) .. controls +(0.5,-0.5) and +(-0.5,0) .. (1.4,-0.2);

    \draw[color = black] (-2,0.8) node{$\T^2$ isotropy};
    \draw[-stealth] (-2,0.7) .. controls +(0.5,-0.5) and +(-0.5,0.2) .. (-1.05,0);
    \draw[-stealth] (-2,0.7) .. controls +(0.5,-0.5) and +(-0.5,0) .. (-0.916,0.5);

    \draw[color = black] (-2, -0.8) node{$\sph^1$ isotropy};
    \draw[-stealth] (-2,-0.7) .. controls +(0.5,0.5) and +(-0.5,-0.5) .. (-0.7071,-0.7071);
    \draw[-stealth] (-2,-0.7) .. controls +(0.5,0.5) and +(-0.5,-0.1) .. (-0.9659,-0.2588);

    \filldraw[black] (A) circle (1pt)
                     (B) circle (1pt)
                     (C) circle (1pt)
                     (D) circle (1pt)
                     (E) circle (1pt)
                     (F) circle (1pt)
                     (G) circle (1pt)
                     (H) circle (1pt)
                     (I) circle (1pt);

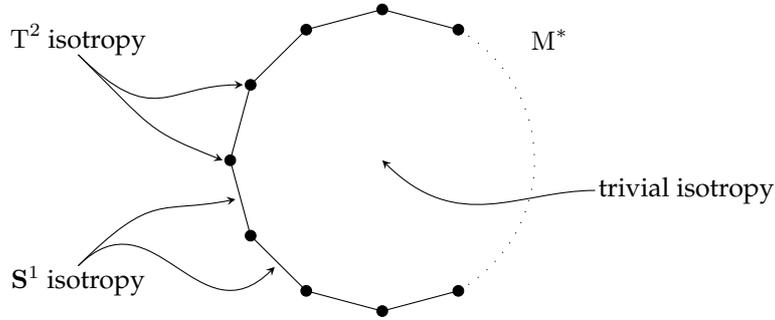
    \draw[color = black] (0, -1.3) node{\underline{Figure 1}: Orbit space of a $\T^n$ action on $\M^{n+2}$};
  \end{tikzpicture}
  \end{center}

Suppose, on the other hand, that $n \geq 2$ and $\M^*$ is a two-dimensional manifold homeomorphic to $D^2$.  Partition the boundary of $\M^*$ into $N$ (ordered) arcs ($N \geq n$) and to each arc assign an $n$-tuple $\underline x_i = (x_{i1}, \dots x_{in}) \in \Z^n$, $i = 1, \dots, N$, where $\gcd(x_{i1}, \dots, x_{in}) = 1$.  Consider each of these $n$-tuples $\underline x_i$ to be the slope of a circle $\G(\underline x_i)$ in $\T^n$.  We say that $\M^*$ is \emph{legally weighted} with weights $\{ \underline x_1, \dots, \underline x_N \}$ if, for any two adjacent slope vectors $\underline x_i$ and $\underline x_{i+1}$ ($i$ cyclic index), there exist $n-2$ other $n$-tuples in $\underline v_1, \dots, \underline v_{n-2} \in \Z^n$ such that the $(n \x n)$-matrix with rows $\underline x_i, \underline x_{i+1}, \underline v_1, \dots, \underline v_{n-2}$ has determinant $\pm 1$.  This is equivalent to saying that the circles $\G(\underline x_i)$ and $\G(\underline x_{i+1})$ have trivial intersection, that is, there exist $n-2$ other circles in $\T^n$ such that together the circles generate (the homology of) $\T^n$.  By \cite{Oh83}, for any legally weighted, orientable $2$-manifold $\M^*$ there is an $(n + 2)$-dimensional manifold $\M$ with orbit space $\M^*$ under an effective $\T^n$ action.

In fact, following similar techniques to those used in \cite{ORa70}, Oh \cite{Oh83} showed that two $(n+2)$-manifolds $\M$ and $\M'$ with smooth, effective $\T^n$ actions are equivariantly diffeomorphic if and only if there is a weight-preserving diffeomorphism between their orbit spaces $\M^*$ and $(\M')^*$.

Therefore, given an $(n+2)$-manifold $\M$ equipped with a smooth, effective $\T^n$ action, we may represent $\M$ uniquely in terms of its weighted orbit space:
\[
\M^{n+2} = \{ \underline x_1, \dots, \underline x_N \}, \ \ N \geq n \geq 2,
\]
where $\underline x_1, \dots, \underline x_N \in \Z^n$ and $\gcd(\underline x_i) = 1$, for $i = 1, \dots, N$.

Moreover, \cite{Oh83} has used the (legally) weighted orbit space $\M^*$ to give information about the fundamental group of the manifold $\M$.  Indeed, if $\M^*$ is a $2$-dimensional disk and there are no exceptional orbits (i.e. no finite isotropy groups), and if $\underline x_i \in \Z^n$ denotes the slope of the $i^{\rm th}$ $\sph^1$ isotropy group $\G(\underline x_i) \subset \T^n$, $i = 1, \dots, N$, then
\beq
\label{Eq:fundgp}
\pi_1 (\M) \subset \Z^n/\<\underline x_1, \dots, \underline x_N\>.
\eeq
In particular, if the $\sph^1$ isotropy groups span $\T^n$ (i.e. some $n$ of the $\underline x_i$ give an $(n \x n)$-matrix with nonzero determinant), then $\pi_1 (\M)$ is finite, and if the $\sph^1$ isotropy groups generate $\T^n$ (i.e. some $n$ of the $\underline x_i$ give an $(n \x n)$-matrix with determinant $\pm 1$), then $\M$ is simply connected.  Note that, if the $\sph^1$ isotropy groups do not span $\T^n$, then $\M$ is diffeomorphic to a product $\M' \x \sph^1$, where $\M'$ is an $(n+1)$-manifold with an effective $\T^{n-1}$ action (cf. \cite{Oh83}).  Hence we will always assume that the $\sph^1$ isotropy groups span $\T^n$.

That the $\sph^1$ isotropy groups give trivial elements of the fundamental group can be explained as follows.  Let $p \in \M$ such that $\T^n \star p$ is a principal orbit and let $q \in \M$ such that $\T^n \star q$ is a singular orbit with isotropy group $\G(\underline x_i)$ of slope $\underline x_i$.  Take a geodesic $\gamma_i (t)$ from $p$ to $q$.  Then $\G(\underline x_i) \star \gamma_i(t)$ is a disk, since $\G(\underline x_i) \star q = q$, and hence $\G(\underline x_i) \star p$ is homotopically trivial.


\section{$\T^2$ actions on $4$-manifolds}
\label{S:dim4}

As discussed in Section \ref{S:cohom2tori}, the orbit space $\M^*$ of a smooth, effective $\T^2$ action on a smooth, compact, simply connected $4$-manifold $\M$ is an oriented $2$-disk whose interior consists of principal orbits and boundary contains $k \geq 2$ isolated points corresponding to orbits with isotropy $\T^2$, i.e. fixed points.  Arcs in the boundary between these points correspond to orbits with circle isotropy groups $G(m_i,n_i)$, $1 \leq i \leq k$, where $\gcd(m_i, n_i) = 1$.

Represent $\M$ in terms of its weighted orbit space, that is
\[
\M=\{\underline x_1 = (m_1,n_1),\ldots, \underline x_k = (m_k,n_k) \}.
\]
In order for an orbit space $\M^*$ to be legally weighted, it is necessary to have
\beq
\label{Eq:legal4dim}
\det \bpm m_{k} & n_k \\ m_1 & n_{1}\epm =\varepsilon_1=\pm 1 \ \ \textrm{and} \ \
\det \bpm m_{i-1} & n_{i-1}\\ m_{i} & n_{i}\epm = \varepsilon_i=\pm 1, \ \ \ i=2,3,\ldots,k.
\eeq
Assume without loss of generality that $\M^*$ is oriented from $\underline x_i = (m_i,n_i)$ to $\underline x_{i+1} = (m_{i+1},n_{i+1})$ and recall that actions on $\M$ are in one-to-one correspondence with weights in the orbit space.

Orlik and Raymond \cite{ORa70} have shown that $\M$ is equivariantly diffeomorphic to a connected sum of $\sph^4$, $\pm\cp^2$ and $\sph^2\x \sph^2$.  However, as we are only interested in $4$-manifolds which admit a Riemannian metric of non-negative curvature and maximal symmetry rank, in particular $\cp^2 \# \pm \cp^2$ and $\sph^2\x \sph^2$, we may restrict our attention to the case where there are exactly four isolated fixed points on the boundary.  This follows from the fact that $\chi(\M) = \chi(\fix(\M,\T^2))$.

Without loss of generality we may reparametrize our $\T^2$ action and assume that the weighted orbit space is of the form shown in Figure 2.


\begin{center}
  \begin{tikzpicture}[scale=1.5]
    \path[coordinate] (1,1)  coordinate(A)
                (-1,1) coordinate(B)
                (-1,-1) coordinate(C)
                (1,-1) coordinate(D);

    \draw (A) -- (B) -- (C) -- (D) -- (A);
    \draw[color = black] (0,1.3) node{$\underline x_1 = (1,0)$};
    \draw[color = black] (-1.7, 0) node{$\underline x_2 = (0,1)$};
    \draw[color = black] (0,-1.3) node{$\underline x_3 = (a,b)$};
    \draw[color = black] (1.7, 0) node{$\underline x_4 = (c,d)$};
    \filldraw[black] (A) circle (1.5pt)
                     (B) circle (1.5pt)
                     (C) circle (1.5pt)
                     (D) circle (1.5pt);
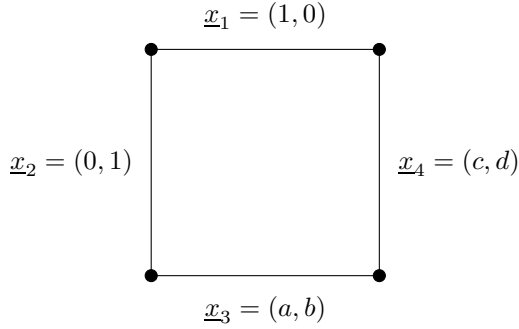
    \draw[color = black] (0, -1.8) node{\underline{Figure 2}: Weighted orbit space for a $\T^2$ action on $\M^4$};
  \end{tikzpicture}
  \end{center}
Since the orbit space is assumed to be legally weighted, it follows from (\ref{Eq:legal4dim}) that $\vep_2 = 1$, $a = -\vep_3$, $d = -\vep_1$ and hence $\vep_1 \vep_3 - bc = \vep_4 = \vep_2 \vep_4$.  Therefore $ad = - 2 \vep_4$ (whenever $\vep_1 \vep_3 = - \vep_2 \vep_4 = -\vep_4$) or $0$ (whenever $\vep_1 \vep_3 = \vep_2 \vep_4 = \vep_4$), from which it follows that $(b,c) = \pm (1,-2 \vep_4)$ or $\pm (- 2 \vep_4,1)$ or $(0, k)$ or $(k, 0)$ for some $k \in \Z$.  From \cite[Section 5]{ORa70}, in the first case $\M$ is diffeomorphic to $\cp^2 \# \cp^2$, while in the latter two cases $\M$ is diffeomorphic to $\cp^2 \# - \cp^2$ or $\sph^2\x \sph^2$, depending on the parity of $k$.  Thus, up to reparametrization and reordering, all possible legally weighted orbit spaces are given in Figure 3 below.

\begin{center}
  \begin{tikzpicture}[scale=1.4]
    \path[coordinate] (-1.5,1)  coordinate(A)
                (-3.5,1) coordinate(B)
                (-3.5,-1) coordinate(C)
                (-1.5,-1) coordinate(D);

    \draw (A) -- (B) -- (C) -- (D) -- (A);
    \draw[color = black] (-2.5,1.3) node{$\underline x_1 = (1,0)$};
    \draw[color = black] (-4.2, 0) node{$\underline x_2 = (0,1)$};
    \draw[color = black] (-2.5,-1.3) node{$\underline x_3 = (1,0)$};
    \draw[color = black] (-0.8, 0) node{$\underline x_4 = (k,1)$};
    \filldraw[black] (A) circle (1.5pt)
                     (B) circle (1.5pt)
                     (C) circle (1.5pt)
                     (D) circle (1.5pt);

    \path[coordinate] (3.5,1)  coordinate(W)
                (1.5,1) coordinate(X)
                (1.5,-1) coordinate(Y)
                (3.5,-1) coordinate(Z);

    \draw (W) -- (X) -- (Y) -- (Z) -- (W);
    \draw[color = black] (2.5,1.3) node{$\underline x_1 = (1,0)$};
    \draw[color = black] (0.8, 0) node{$\underline x_2 = (0,1)$};
    \draw[color = black] (2.5,-1.3) node{$\underline x_3 = (1,1)$};
    \draw[color = black] (4.2, 0) node{$\underline x_4 = (2,1)$};
    \filldraw[black] (W) circle (1.5pt)
                     (X) circle (1.5pt)
                     (Y) circle (1.5pt)
                     (Z) circle (1.5pt);

    \draw[color = black] (-2.5, -1.7) node{$\sph^2 \x \sph^2$ (if $k$ even) and };
    \draw[color = black] (-2.5, -2) node{$\cp^2 \# -\cp^2$ (if $k$ odd)};
    \draw[color = black] (2.5, -1.7) node{$\cp^2 \# \cp^2$};
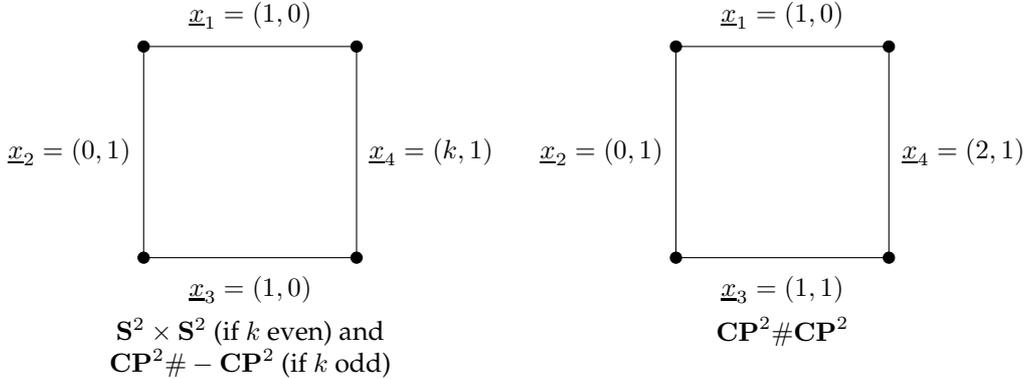
    \draw[color = black] (0, -2.5) node{\underline{Figure 3}: Possible legally weighted orbit spaces for $\cp^2 \# \pm \cp^2$ and $\sph^2 \x \sph^2$};
  \end{tikzpicture}
  \end{center}


\section{$\T^3$ actions on $5$-manifolds}
\label{S:dim5}

Let $\M$ be a smooth, compact, simply connected $5$-manifold with a smooth, effective $\T^3$ action.  By the discussion in Section \ref{S:cohom2tori}, the orbit space $\M^*$ of the action is an oriented $2$-disk whose interior consists of principal orbits and its boundary contains $k\geq 3$ isolated points corresponding to orbits with isotropy $\T^2$.  As before, arcs in the boundary between these points correspond to orbits with circle isotropy groups $G(\ell_i, m_i, n_i)$, $2 \leq i \leq k$, where $\gcd(\ell_i, m_i, n_i) = 1$, and $\M$ may be represented in terms of its weighted orbit space, i.e.
\[
\M=\{\underline{x}_1 = (\ell_1, m_1, n_1),\ldots,\underline{x}_k = (\ell_k, m_k, n_k)\}.
\]

Oh \cite{Oh83} has classified such manifolds up to equivariant diffeomorphism.  Indeed, if $k = 3$, $\M$ (together with its $\T^3$ action) is equivariantly diffeomorphic to $\sph^5$ equipped with a linear action, while, for $k \geq 4$, $\M$ is equivariantly diffeomorphic to either $\#(k-3)(\sph^3 \x \sph^2)$ or $(\sph^3 \tilde{\x} \sph^2)\#(k-4)(\sph^3 \x \sph^2)$, depending respectively on whether the second Stiefel-Whitney class $w_2 (\M)$ is trivial or not.

As manifolds which admit a Riemannian metric with non-negative curvature and maximal symmetry rank are the focus of our attention, by \cite{GGS10} we may therefore assume that $k = 4$, namely that $\M$ is one of $\sph^3 \x \sph^2$ or $\sph^3 \tilde{\x} \sph^2$.

By a suitable reparametrization of the $\T^3$ action, if necessary, we may assume that the orbit space of these manifolds looks like the one in Figure 4 below.


\begin{center}
  \begin{tikzpicture}[scale=1.5]
    \path[coordinate] (1,1)  coordinate(A)
                (-1,1) coordinate(B)
                (-1,-1) coordinate(C)
                (1,-1) coordinate(D);

    \draw (A) -- (B) -- (C) -- (D) -- (A);
    \draw[color = black] (0,1.3) node{$\underline x_1 = (1,0,0)$};
    \draw[color = black] (-1.8, 0) node{$\underline x_2 = (0,1,0)$};
    \draw[color = black] (0,-1.3) node{$\underline x_3 = (p,q,r)$};
    \draw[color = black] (1.8, 0) node{$\underline x_4 = (x,y,z)$};
    \filldraw[black] (A) circle (1.5pt)
                     (B) circle (1.5pt)
                     (C) circle (1.5pt)
                     (D) circle (1.5pt);

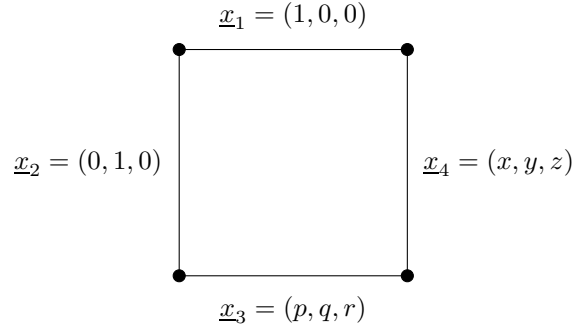
    \draw[color = black] (0, -1.8) node{\underline{Figure 4}: Weighted orbit space for a $\T^3$ action on $\M^5$};
  \end{tikzpicture}
  \end{center}

Notice that $\underline x_1 = (1,0,0)$, $\underline x_2 = (0,1,0)$, $\underline x_3 = (p,q,r)$ and $\underline x_4 = (x,y,z)$ generate $\<(1,0,0), (0,1,0), (0, 0, \gcd(r,z)) \> \subset \Z^3$.  Therefore, from (\ref{Eq:fundgp}) we deduce that $\pi_1 (\M) \subset \Z_{\gcd(r,z)}$.  In fact, as the following proposition shows, we can do better.


\begin{prop}
\label{Prop:fundgp}
If $\T^3$ acts smoothly and effectively on a smooth, compact $5$-manifold $\M$ with (legally) weighted orbit space as in Figure 4, then $\pi_1 (\M) = \Z_{\gcd(r,z)}$.  In particular, if $\M$ is simply connected, then $\gcd(r,z) = 1$.
\end{prop}


\begin{proof}
We want to use the theorem of van Kampen.  Let $b \in \M$ lie on a principal orbit and let $\{\underline e_1, \underline e_2, \underline e_3)\}$ denote the standard basis in $\R^3$.  The circle of slope $(\ell_1, \ell_2, \ell_3)$ in $\T^3$ is, as always, denoted by $\G(\ell_1, \ell_2, \ell_3)$.  Set $\G_i := \G(\underline e_i)$, $i = 1,2,3$.  In particular, from \cite{Oh83} (cf. (\ref{Eq:fundgp})) we know that $\pi_1 (\M, b)$ is generated by the homotopy classes (of circle orbits) $[\G_1 \star b]$, $[\G_2 \star b]$ and $[\G_3 \star b]$.  Denote the arcs in $\M^*$ with weights $(0,1,0)$ and $(x,y,z)$ by $L_1^*$ and $L_2^*$ respectively, and their respective pre-images under the $\T^3$ action by $L_1$ and $L_2$.

Divide $\M$ into two open regions $\Omega_1$ and $\Omega_2$, centred around $L_1$ and $L_2$ respectively, and with $b \in \Omega_1 \cap \Omega_2$.  We may assume that $\Omega_i$ is a $2$-disk bundle over the corresponding arc $L_i$, $i = 1,2$, since the isotropy group $\G_3$ (resp. $\G(x,y,z)$) acts freely away from $L_1$ (resp. $L_2$) (see \cite[Chapter VI, Theorem 2.2]{Br}).  Then the orbit space of $\M$ under the $\T^3$ action is as in Figure 5.
\begin{center}
  \begin{tikzpicture}[scale=1.4]
    \path[coordinate] (1,1)  coordinate(A)
                (-1,1) coordinate(B)
                (-1,-1) coordinate(C)
                (1,-1) coordinate(D)
                (0,0) coordinate(b)
                (0.3, 1) coordinate(ur)
                (-0.3, 1) coordinate(ul)
                (0.3, -1) coordinate(dr)
                (-0.3, -1) coordinate(dl);

    \fill[color = gray!30!white] (ul) -- (B) -- (C) -- (dl) -- cycle;
    \fill[color = gray!30!white] (ur) -- (A) -- (D) -- (dr) -- cycle;
    \fill[color = gray!50!white] (ur) -- (ul) -- (dl) -- (dr) -- cycle;

    \draw (A) -- (B)
          (C) -- (D);
    \draw[ultra thick] (B) -- (C)
                       (D) -- (A);
    \draw[color = black] (0,1.3) node{$\G_1$};
    \draw[color = black] (-1.3, -0.5) node{$\G_2$};
    \draw[color = black] (0,-1.3) node{$\G(p,q,r)$};
    \draw[color = black] (1.6, -0.5) node{$\G(x,y,z)$};
    \draw[color = black] (0.15, 0.15) node{$b^*$};
    \filldraw[black] (A) circle (1.5pt)
                     (B) circle (1.5pt)
                     (C) circle (1.5pt)
                     (D) circle (1.5pt)
                     (b) circle (1pt);

    \draw[loosely dashed] (ur) -- (dr)
                          (ul) -- (dl);
    \draw[color = black] (-1.3, 0.5) node{$L_1^*$};
    \draw[color = black] (1.3, 0.5) node{$L_2^*$};
    \draw[color = black] (-0.6, 0.5) node{$\Omega_1^*$};
    \draw[color = black] (0.6, 0.5) node{$\Omega_2^*$};

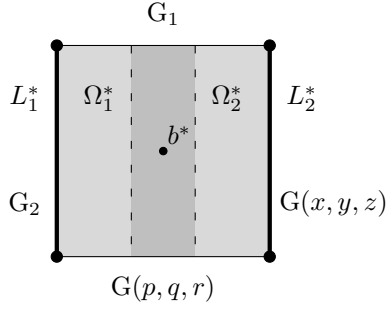
    \draw[color = black] (0, -1.8) node{\underline{Figure 5}: Van Kampen decomposition of $\M$};

  \end{tikzpicture}
\end{center}

The points in the interior of the arc $L_1^*$ correspond to points in $\M$ with isotropy $\G_2$, while the endpoints of $L_1^*$ correspond to points in $\M$ with isotropy given by $\G_1 \x \G_2$ and $\G_2 \x \G(p,q,r)$ respectively.  We can therefore think of $L_1^*$ as the quotient of a $\G_1 \x \G_3 = \T^3 / \G_2$ action on the $3$-manifold $L_1$, where the interior of $L_1^*$ corresponds to points with trivial isotropy and the endpoints to points with isotropy $\G(1,0)$ and $\G(p,r)$ respectively.  By the work of Orlik and Raymond \cite{ORa70}, we know that $L_1$ is a lens space $L(r;p)$.

As $\Omega_1$ is a $D^2$-bundle over $L_1$, and since $\G_1 \star b$ and $\G_2 \star b$ are homotopically trivial in $\Omega_1$, it follows that $\pi_1 (\Omega_1,b) \cong \pi_1 (L_1) = \<[\G_3 \star b]\> / \<p [\G_3 \star b]\> = \Z_r$.

Since $\M^*$ is legally weighted, there exist $\lambda, \mu \in \Z$ such that $\lambda y + \mu z = 1$.  Let $A : \R^3 \to \R^3$ be the linear map given by
$$
A = \bpm
    1 & -\lambda x & -\mu x \\
    0 & z & -y \\
    0 & \lambda & \mu
    \epm.
$$
Since $\det A = 1$, it follows that we can reparametrize our $\T^3$ action by $A$, that is, $A \underline e_i$, $i = 1,2,3$, generate $\T^3$.  Notice in particular that
$$
A \underline e_1 = \underline e_1, \quad
A \bpm x \\ y \\ z \epm = \underline e_3 \quad {\rm and} \quad
A \bpm p \\ q \\ r \epm = \bpm p - (\lambda q + \mu r) x \\ qz - ry \\ \lambda q + \mu r \epm.
$$
In terms of the new parameters, denote circles of slope $\underline e_i$ and $(p - (\lambda q + \mu r) x, qz - ry, \lambda q + \mu r)$ by $\tilde \G_i$ and $\tilde \G'$ respectively.  Then we may relabel $\Omega_2^*$ as in Figure 6 below.

\begin{center}
  \begin{tikzpicture}[scale=1.4]
    \path[coordinate] (1,1)  coordinate(A)
                (1,-1) coordinate(D)
                (0,0) coordinate(b)
                (-0.3, 1) coordinate(ul)
                (-0.3, -1) coordinate(dl);

    \draw (A) -- (ul)
          (dl) -- (D);
    \draw[ultra thick]  (D) -- (A);
    \draw[color = black] (0.3,1.3) node{$\tilde \G_1$};
    \draw[color = black] (0.3,-1.3) node{$\tilde \G'$};
    \draw[color = black] (1.3, -0.5) node{$\tilde \G_3$};
    \draw[color = black] (0.15, 0.15) node{$b^*$};
    \filldraw[black] (A) circle (1.5pt)
                     (D) circle (1.5pt)
                     (b) circle (1pt);

    \draw[loosely dashed] (ul) -- (dl);
    \draw[color = black] (1.3, 0.5) node{$L_2^*$};
    \draw[color = black] (0.6, 0.5) node{$\Omega_2^*$};

    \draw[color = black] (0, -1.8) node{\underline{Figure 6}: $\Omega_2^*$ after reparametrization};

  \end{tikzpicture}
\end{center}
Hence, as for $L_1$, we deduce that $L_2$ is a lens space $L(qz - ry; p - (\lambda q + \mu r) x)$ and therefore $\pi_1 (\Omega_2, b) \cong \pi_1 (L_2) = \<[\tilde \G_3 \star b]\> / \<(qz-ry) [\tilde \G_3 \star b]\> = \Z_{qz-ry}$, where $\lambda y + \mu z = 1$.  In our original coordinate system it follows that $\pi_1 (\Omega_2, b) \cong \<[\G(0, \mu, -\lambda) \star b]\> / \<(qz-ry) [\G(0, \mu, -\lambda) \star b]\>$, since $A(0, \mu, -\lambda)^t = \underline e_2$.

Consider now $\Omega_1 \cap \Omega_2$.  By the same reasoning as in \cite{Oh83}, $\pi_1 (\Omega_1 \cap \Omega_2, b)$ is generated by the homotopy classes $[\G_1 \star b]$, $[\G_2 \star b]$ and $[\G_3 \star b]$.  Furthermore, $[\G_1 \star b] = [\G(p,q,r) \star b] = 0 \in \pi_1 (\Omega_1 \cap \Omega_2, b)$.

The inclusion maps induce homomorphisms $\iota_i:\pi_1 (\Omega_1 \cap \Omega_2, b) \to \pi_1 (\Omega_i, b)$, $i = 1,2$.  It is clear from our previous discussion of $\pi_1 (\Omega_1, b)$ that $\iota_1 ([\G_3 \star b]) = [\G_3 \star b]$ and $\iota_1 ([\G_2 \star b]) = [\G_2 \star b] = 0$.  On the other hand, since $A^{-1}$ gives a reparametrization of $\T^3$,
\begin{align*}
\iota_2([\G_3 \star b]) &= \iota_2 (-y [\G(0, \mu, -\lambda) \star b] + \lambda [\G(x,y,z) \star b] - \mu x [\G_1 \star b]) \\
&= -y \iota_2([\G(0, \mu, -\lambda) \star b]) + \lambda \iota_2([\G(x,y,z) \star b]) \\
&= -y [\G(0, \mu, -\lambda) \star b] + \lambda [\G(x,y,z) \star b] \\
&= -y [\G(0, \mu, -\lambda) \star b]
\end{align*}
since $[\G(x,y,z) \star b] = 0 \in \pi_1 (\Omega_2, b)$.  Similarly $\iota_2([\G_2 \star b]) = z [\G(0, \mu, -\lambda) \star b]$.

Let $[\alpha] = [\G_3 \star b] \in \pi_1 (\Omega_1, b)$ and $[\beta] = [\G(0, \mu, -\lambda) \star b] \in \pi_1 (\Omega_2, b)$.  Then, by van Kampen's Theorem
\begin{align*}
\pi_1 (\M, b) &= (\pi_1 (\Omega_1, b) * \pi_1 (\Omega_2, b)) / \<\iota_1([\G_i \star b]) = \iota_2([\G_i \star b]), \ i = 2, 3\> \\
&= \<[\alpha], [\beta] \mid r[\alpha] = 0, (qz-ry)[\beta] = 0\> / \<[\alpha] = - y [\beta], \ z [\beta] = 0\> \\
&= \<[\beta] \mid z [\beta] = 0, ry [\beta] = 0\> \\
&= \<[\beta] \mid \gcd(z, ry) [\beta] = 0\> \\
&= \Z_{\gcd(r, z)}, \qquad \quad {\rm since} \ \gcd(y,z) = 1.
\end{align*}
\end{proof}


\begin{cor}
\label{Cor:fundgp}
If $\T^3$ acts smoothly and effectively on a smooth, compact, simply connected $5$-manifold $\M^5$ with (legally) weighted orbit space as in Figure 4, then
\beq
\label{Eq:gcdCond}
\gcd(r,z) = 1, \ \ \gcd(p,r) = 1, \ \ \gcd(y,z) = 1 \ \  \text{and} \ \ \gcd(py-qx, rx - pz, qz - ry) = 1.
\eeq
\end{cor}


\begin{proof}
The fact that $\gcd(r,z) = 1$ follows directly from Proposition \ref{Prop:fundgp}.  From the discussion in Section \ref{S:cohom2tori} the orbit space is legally weighted if and only if there exist triples $\underline y_i = (\alpha_i, \beta_i, \gamma_i) \in \Z^3$, $i = 1,2,3,4$, such that
$$
\det \bpm \underline y_1 \\ \underline x_1 \\ \underline x_2 \epm = \pm 1, \ \
\det \bpm \underline y_2 \\ \underline x_2 \\ \underline x_3 \epm = \pm 1, \ \
\det \bpm \underline y_3 \\ \underline x_3 \\ \underline x_4 \epm = \pm 1 \ \ {\rm and} \ \
\det \bpm \underline y_4 \\ \underline x_4 \\ \underline x_1 \epm = \pm 1.
$$
This is possible if and only if the other three $\gcd$ conditions hold.
\end{proof}

Suppose that we have a weighted orbit space as in Figure 4 and that the conditions in (\ref{Eq:gcdCond}) hold.  Set $a = z$, $b = -(rx - pz)$, $c = r$ and $d = qz - ry$.  Then $\gcd(a,c) = 1$, $\gcd(a, d) = 1$ and $\gcd(b, c) = 1$.  Let $m, n \in \Z$ such that $am + cn = 1$.

Then $(am + cn) b = b = - cx + ap$, hence $c(x + bn) = a(p - bm)$.  Therefore, since $\gcd(a,c) = 1$, $x = -bn - ak$ and $p = bm - ck$, for some $k \in \Z$.  Similarly, it follows that $y = -dn - a\ell$ and $q = dm - c\ell$, for some $\ell \in \Z$.  Further, we find that $\gcd(b,d,py-qx) = 1$ implies that $\gcd(b,d) = 1$.

Therefore all possible legally weighted orbit spaces as in Figure 4 for a smooth, effective $\T^3$ action on the simply connected manifolds $\sph^3 \x \sph^2$ and $\sph^3 \tilde\x \sph^2$ are given by the diagram

  \begin{center}
  \begin{tikzpicture}[scale=1.4]
    \path[coordinate] (1,1)  coordinate(A)
                (-1,1) coordinate(B)
                (-1,-1) coordinate(C)
                (1,-1) coordinate(D);

    \draw (A) -- (B) -- (C) -- (D) -- (A);
    \draw[color = black] (0,1.3) node{$\underline x_1 = (1,0,0)$};
    \draw[color = black] (-1.8, 0) node{$\underline x_2 = (0,1,0)$};
    \draw[color = black] (0,-1.3) node{$\underline x_3 = (b m - c k, d m - c \ell, c)$};
    \draw[color = black] (2.6, 0) node{$\underline x_4 = (-b n - a k, -d n - a \ell, a)$};
    \filldraw[black] (A) circle (1.5pt)
                     (B) circle (1.5pt)
                     (C) circle (1.5pt)
                     (D) circle (1.5pt);

    \draw[color = black] (0, -1.8) node{\underline{Figure 7}: Possible legally weighted orbit spaces for $\sph^3 \x \sph^2$ and $\sph^3 \tilde\x \sph^2$};

  \end{tikzpicture}
  \end{center}
where $\gcd(a,c) = \gcd(a,d) = \gcd(b,c) = \gcd(b,d) = 1$, hence $\gcd(ab,cd) = 1$, and $k, \ell, m, n \in \Z$ with $am + cn = 1$.


\section{Free actions on $\sph^3 \x \sph^3$}
\label{S:Biquotients}

Recall that $\sph^3 \x \sph^3$ is a Lie group consisting of pairs of unit quaternions, i.e. pairs $\left(\bsm q_1 \\ q_2 \esm\right)$ where $q_n = \alpha_n + \beta_n j$, with $\alpha_n, \beta_n \in \C$, $|\alpha_n|^2 + |\beta_n|^2 = 1$, for $n = 1,2$.  Recall that $ij = - ji$ and so, in particular, $\beta j = j \bar \beta$ for all $\beta \in \C$.

Define $\sph^1_j := \{\beta j \ | \ \beta \in \C, \ |\beta| = 1\} \subset \sph^3$.  The image under the quotient map of points $\left(\bsm q_1 \\ q_2 \esm \right) \in \sph^3 \x \sph^3$ and subgroups $\HH \subset \sph^3 \x \sph^3$ will always be denoted by $\left[\bsm q_1 \\ q_2 \esm \right]$ and $[\HH]$ respectively.

Let $\met_0$ be a bi-invariant metric on $\sph^3 \x \sph^3$.  The usual isometric $\T^4$ action on $(\sph^3 \x \sph^3, \met_0)$, namely
$$
(u,v,w,z) \star \bpm q_1 \\ q_2 \epm =
 \bpm u q_1 \bar v \\ w q_2 \bar z \epm, \quad (u,v,w,z) \in \T^4, \ q_1, q_2 \in \sph^3,
$$
is not effective.  Consider instead the $\T^4$ action given by
\beq
\label{Eq:EffT4}
(u,v,w,z) \star \bpm q_1 \\ q_2 \epm =
 \bpm u \alpha_1 + v \beta_1 j \\
     w \alpha_2 + z \beta_2 j \epm,
\quad (u,v,w,z) \in \T^4, \ q_1, q_2 \in \sph^3.
\eeq
This action is clearly effective.  Moreover, it is an isometric action since it may be rewritten as a (well-defined) two-sided action as follows:
$$
(u,v,w,z) \star \bpm q_1 \\ q_2 \epm =
 \bpm u^\frac{1}{2} v^\frac{1}{2} q_1 u^\frac{1}{2} \bar v^\frac{1}{2} \\
      w^\frac{1}{2} z^\frac{1}{2} q_2 w^\frac{1}{2} \bar z^\frac{1}{2} \epm.
$$
\begin{rem}
Define $\SUU{2} := \{(A,B) \in \U(2) \x \U(2) \mid \det (A) = \det (B)\}$.  Then the action (\ref{Eq:EffT4}) can, in fact, be thought of as an action of $\T^4 \subset \SUU{2}^2$ on $\SU(2) \x \SU(2)$ since, for example, an $\sph^1$ action $z \star q = z^{\frac{2k+1}{2}} q \bar z^{\frac{2\ell+1}{2}}$, $q = \alpha + \beta j \in \sph^3$, $z \in \sph^1$, is equivalent to an action of $\sph^1 \subset \SUU{2}$ on $\SU(2)$ via
$$
z \star \bpm \alpha & \beta \\ - \bar \beta & \bar \alpha \epm =
\bpm z^k & \\ & \bar z^{k+1} \epm
\bpm \alpha & \beta \\ - \bar \beta & \bar \alpha\epm
\bpm \bar z^\ell & \\ & z^{\ell +1}\epm.
$$
\end{rem}
\vspace{3mm}

Consider the tuples $(a,b,c,d), (-n,k,m,\ell) \in \Z^4$, such that $am + cn = 1$.  As
$$
\det \bpm 0 & 1 & 0 & 0 \\
          0 & 0 & 0 & 1 \\
          -n & k & m & \ell \\
          a & b & c & d \epm = am + cn = 1
$$
we may reparametrize the effective, isometric action (\ref{Eq:EffT4}) via
$$
(u,v,w,z) \mapsto
(z^a \bar w^n, u z^b w^k, z^c w^m, v z^d w^\ell)
$$
to give an effective, isometric $\T^4$ action on $(\sph^3 \x \sph^3, \met_0)$ via
\beq
\label{Eq:effT4}
(u,v,w,z) \star \bpm q_1 \\ q_2 \epm =
\bpm z^a \bar w^n \alpha_1 + u z^b w^k\beta_1 j \\
     z^c w^m \alpha_2 + v z^d w^\ell \beta_2 j \epm.
\eeq
In particular, if the circle (resp. torus) given by the $z$-coordinate (resp. $(w,z)$-coordinates) acts freely on $\sph^3 \x \sph^3$, then there is an induced effective, isometric action on the quotient by the $(u,v,w)$-torus (resp. $(u,v)$-torus).

This raises the question of when the $z$-circle and $(w,z)$-torus act freely.  It is easy to check that $\sph^1$ acts freely on $\sph^3 \x \sph^3$ via
\beq
\label{Eq:freeS1}
z \star \bpm q_1 \\ q_2 \epm =
\bpm z^a \alpha_1 + z^b \beta_1 j \\
     z^c \alpha_2 + z^d \beta_2 j \epm
\eeq
if and only if $\gcd(a,c) = \gcd(a,d) = \gcd(b,c) = \gcd(b,d) =1$, that is, if and only if $\gcd(ab,cd) = 1$.  We denote the quotient $(\sph^3 \x \sph^3, \met_0) \bq \sph^1$ by $\M^5_{a,b,c,d}$.

DeVito \cite{DV} has classified up to diffeomorphism all possible biquotients $(\sph^3 \x \sph^3) \bq \sph^1$.  More precisely, he has shown that only $\sph^3 \x \sph^2$ and $\sph^3 \tilde\x \sph^2$ can arise.  Since both of these manifolds satisfy $H_2(\M^5; \Z) = \Z$, this classification amounts to computing the second Stiefel-Whitney class $w_2((\sph^3 \x \sph^3) \bq \sph^1) \in H^2 (\M^5; \Z_2)$ (cf. \cite{Ba}).  A consequence of DeVito's classification is that the quotient resulting from the explicit displayed action suggested by Pavlov \cite{Pav04} and claimed to be $\sph^3 \tilde\x \sph^2$, in fact has $w_2 = 0$, hence must be diffeomorphic to $\sph^3 \x \sph^2$.  In the paragraph preceeding this displayed action, Pavlov describes in words an action which gives the correct quotient.

It can be shown that $w_2 (\M^5_{a,b,c,d}) = a + b + c + d \in \Z_2$.  Hence, together with the freeness condition, it follows that $\M^5_{a,b,c,d}$ is diffeomorphic to $\sph^3 \tilde\x \sph^2$ if and only if exactly one of $a$, $b$, $c$ or $d$ is even (and is diffeomorphic to $\sph^3 \x \sph^2$ otherwise).  It should be noted that although we have parametrized our $\sph^1$ actions on $\sph^3 \x \sph^3$ differently to those in \cite{DV}, the computation of the Stiefel-Whitney class $w_2$ is completely analogous to the computation carried out therein, and hence has been omitted.

On the other hand, $\T^2$ acts freely on $\sph^3 \x \sph^3$ via
\beq
\label{Eq:freeT2}
(w,z) \star \bpm q_1 \\ q_2 \epm =
\bpm z^a \bar w^n \alpha_1 + z^b w^k\beta_1 j \\
     z^c w^m \alpha_2 + z^d w^\ell \beta_2 j \epm
\eeq
if and only if (without loss of generality)
\beq
\label{Eq:T2freeness}
am + cn = 1, \ \ a \ell + d n = \vep_2, \ \ b m - c k = \vep_3 \ \ {\rm and} \ \ b \ell - d k = \vep_4,
\eeq
where $\vep_2, \vep_3, \vep_4 = \pm 1$.

\begin{lem}
\label{lem:relprime}
Suppose $a, c \in \Z$ are relatively prime and let $m_0, n_0 \in \Z$ such that $a m_0 + c n_0 = 1$.  Then $m,n \in \Z$ satisfy $a m + c n = 1$ if and only if there exists $x \in \Z$ such that $m = m_0 - cx$ and $n = n_0 + ax$.
\end{lem}

\begin{proof}
As the cases $a = 0$ and $c = 0$ are trivial, assume that $a,c \neq 0$.  By subtracting $a m + c n = 1$ from $a m_0 + c n_0 = 1$ it follows that $a(m_0 - m) = - c(n_0 - n)$.  But $a$ and $c$ are assumed to be relatively prime, hence $n_0 - n = - ax$ and $m_0 - m = cx$ for some $x \in \Z$.
\end{proof}

Using this lemma, one can describe all possible free $\T^2$ actions on $\sph^3 \x \sph^3$.  In the case where $\vep_2 \vep_3 \vep_4 = 1$ it turns out that, up to reparametrization of the action and diffeomorphisms of $\sph^3 \x \sph^3$ (namely $q_1 \leftrightarrow q_2$ and $q_i \mapsto \bar q_i$), the action has the form
\beq
\label{Eq:inffreeT2}
(w,z) \star \bpm q_1 \\ q_2 \epm =
\bpm z \bar w^r \alpha_1 + z w^{r+\lambda} \beta_1 j \\
     w \alpha_2 + w \beta_2 j \epm
\eeq
where $r \in \Z$ and $\lambda \in \{0,1\}$.  Furthermore, $\M^4_{r,\lambda} := (\sph^3 \x \sph^3, \met_0) \bq \T^2$ is diffeomorphic to $\sph^2 \x \sph^2$ for $\lambda = 0$, and to $\cp^2 \# - \cp^2$ for $\lambda = 1$.  This follows from work in \cite{DV} via a change of parameters (cf. also \cite{Ch}, \cite{To}).

In the case where $\vep_2 \vep_3 \vep_4 = -1$ it can be shown that, up to reparametrization of the action and diffeomorphisms of $\sph^3 \x \sph^3$, the action has the form
\beq
\label{Eq:unifreeT2}
(w,z) \star \bpm q_1 \\ q_2 \epm =
\bpm z \alpha_1 + w \beta_1 j \\
     \bar z w \alpha_2 + z w \beta_2 j \epm.
\eeq
The quotient $\M^4 = (\sph^3 \x \sph^3, \met_0) \bq \T^2$ is diffeomorphic to $\cp^2 \# \cp^2$ (cf. \cite{To}, \cite{DV}).
\begin{rem}
Recall that the only smooth, compact, simply connected $4$-manifolds admitting both non-negative curvature and an isometric circle action are $\sph^4$, $\cp^2$, $\cp^2 \# \pm\cp^2$ and $\sph^2 \x \sph^2$ (cf. \cite{Kl}, \cite{SY}).  The long exact homotopy sequence for a fibration $\T^2 \to \sph^3 \x \sph^3 \to \M^4$ shows that $\sph^4$ and $\cp^2$ cannot arise as biquotients $(\sph^3 \x \sph^3, \met_0) \bq \T^2$.
\end{rem}

\section{Torus actions on four-dimensional biquotients}
\label{S:4dim}

In the $\T^4$ action described by (\ref{Eq:effT4}), let $\T^2_{uv}$ and $\T^2_{wz}$ denote the $2$-tori given by the $(u,v)$ and $(w,z)$ coordinates respectively.  Consider the induced effective, isometric $\T^2_{uv}$ action on any $\M^4 = (\sph^3 \x \sph^3) \bq \T^2_{wz}$, where $\T^2_{wz}$ acts freely (as in (\ref{Eq:inffreeT2}) or (\ref{Eq:unifreeT2})), namely
\beq
\label{Eq:uvact}
(u,v) \star \bbm q_1 \\ q_2 \ebm =
\bbm \alpha_1 + u \beta_1 j \\ \alpha_2 + v \beta_2 j \ebm.
\eeq

It is clear that the quotient of $\sph^1 \x \sph^1 \subset \sph^3 \x \sph^3$ under $\T^2_{wz}$ is a point, whereas $\sph^1 \x \sph^3 \subset \sph^3 \x \sph^3$ has quotient diffeomorphic to $\sph^2$.  Similarly, $\sph^1 \x \sph^1_j$, $\sph^1_j \x \sph^1$ and $\sph^1_j \x \sph^1_j$ quotient to points, while $\sph^1_j \x \sph^3$, $\sph^3 \x \sph^1$ and $\sph^3 \x \sph^1_j$ have quotients diffeomorphic to $\sph^2$.

For each of the biquotients $(\sph^3 \x \sph^3, \met_0) \bq \T^2_{wz}$ we will determine the fixed-point set of the corresponding effective, isometric $\T^2_{uv}$ action, as well as any additional isotropy that may arise.  As mentioned in Section \ref{S:cohom2tori}, the only possible isotropy groups of an effective $\T^2$ action on a four-dimensional manifold are $\sph^1$ and $\T^2$.

\begin{lem}
\label{lem:T2isotropy}
Let $\T^2_{wz}$ act on $(\sph^3 \x \sph^3, \met_0)$ via (\ref{Eq:inffreeT2}) and let $\T^2_{uv}$ act on $\M^4_{r,\lambda} = (\sph^3 \x \sph^3, \met_0) \bq \T^2_{wz}$ via (\ref{Eq:uvact}).  Then the fixed points of the action are the four points $[\sph^1 \x \sph^1]$, $[\sph^1_j \x \sph^1]$, $[\sph^1 \x \sph^1_j]$, $[\sph^1_j \x \sph^1_j] \in \sph^3 \x \sph^3 \bq \T^2_{k, \vep}$, while a point has $\sph^1$ isotropy if and only if it lies in one of the four two-spheres $[\sph^1 \x \sph^3]$, $[\sph^1_j \x \sph^3]$, $[\sph^3 \x \sph^1]$ or $[\sph^3 \x \sph^1_j]$.  In particular, the $\sph^1$ isotropy subgroups of $\T^2_{uv}$ are arranged according to the diagram
\beq
\label{fig:dim4isotropy}
\xymatrix{
[\sph^1 \x \sph^1] \ar@{.}[rr]^{\{(u,1)\}}_{[\sph^1 \x \sph^3]} \ar@{.}[dd]_{\{(1,v)\}}^{[\sph^3 \x \sph^1]} & &
[\sph^1 \x \sph^1_j] \ar@{.}[dd]^{\{(v^{2r+\lambda}, v)\}}_{[\sph^3 \x \sph^1_j]} \\ & & \\
[\sph^1_j \x \sph^1] \ar@{.}[rr]_{\{(u,1)\}}^{[\sph^1_j \x \sph^3]} & &
[\sph^1_j \x \sph^1_j]
}
\eeq
\end{lem}

\begin{proof}
Suppose that $\left[\bsm q_1 \\ q_2 \esm \right] \in \M^4_{r, \lambda}$ is fixed by some element of $\T^2_{uv}$.  That is, there exist $(w,z) \in \T^2_{wz}$ such that
$$
\bpm \alpha_1 + \beta_1 j \\ \alpha_2 + \beta_2 j \epm =
\bpm q_1 \\ q_2 \epm =
\bpm z \bar w^r \alpha_1 + u z w^{r + \lambda} \beta_1 j \\
     w \alpha_2 + w v \beta_2 j \epm.
$$
Then each of the following four conditions must hold:
\begin{itemize}
\item $\alpha_1 = 0$ or $z \bar w^r = 1$;
\item $\beta_1 = 0$ or $u z w^{r + \lambda} = 1$;
\item $\alpha_2 = 0$ or $w = 1$;
\item $\beta_2 = 0$ or $w v = 1$.
\end{itemize}
First suppose that $\alpha_1 = \alpha_2 = 0$.  Then $w = \bar v$ and $1 = u z w^{r + \lambda} = u z \bar v^{r + \lambda}$, since $|\beta_1| = |\beta_2| = 1$.  Therefore $z = \bar u v^{r + \lambda}$ and hence for all $(u,v)$ we have $(u,v) \star [\sph^1_j \x \sph^1_j] = [\sph^1_j \x \sph^1_j]$. That is, $[\sph^1_j \x \sph^1_j]$ is a fixed point of the $\T^2_{uv}$ action.

The analogous computations in the cases $\beta_1 = \alpha_2 = 0$, $\alpha_1 = \beta_2 = 0$ and $\beta_1 = \beta_2 = 0$ show that $[\sph^1 \x \sph^1_j]$, $[\sph^1_j \x \sph^1]$ and $[\sph^1 \x \sph^1]$, respectively, are also fixed points of the $\T^2_{uv}$ action.

By effectiveness of the $\T^2_{uv}$ action, it is clear that whenever all of $\alpha_1$, $\alpha_2$, $\beta_1$ and $\beta_2$ are non-zero, the isotropy group is trivial.

Now suppose that $\alpha_2 = 0$ and $\alpha_1, \beta_1 \neq 0$.  Then $w = \bar v$ and $z \bar w^r = 1 = u z w^{r + \lambda}$.  Hence $z = \bar v^r = \bar u v^{r + \lambda}$, from which it follows that $u = v^{2r + \lambda}$.  Therefore each point of $[\sph^3 \x \sph^1_j]$ is fixed by the circle $\{(v^{2r + \lambda}, v)\}$ in $\T^2_{uv}$.

Similarly we find that the points of $[\sph^1 \x \sph^3]$, $[\sph^1_j \x \sph^3]$ and $[\sph^3 \x \sph^1]$ are fixed by the circles $\{(u, 1)\}$, $\{(u, 1)\}$ and $\{(1, v)\}$ in $\T^2_{uv}$ respectively.
\end{proof}

\begin{prop}
\label{prop:4mnfds}
Every smooth, effective $\T^2$ action on $\sph^2 \x \sph^2$, $\cp^2 \# - \cp^2$ or $\cp^2 \# \cp^2$ is equivariantly diffeomorphic to an effective, isometric $\T^2_{uv}$ action on the corresponding biquotient $(\sph^3 \x \sph^3, \met_0) \bq \T^2_{wz}$.
\end{prop}

\begin{proof}
From the discussion in Section \ref{S:dim4} we know that there is a unique smooth, effective $\T^2$ action on $\cp^2 \# \cp^2$ up to equivariant diffeomorphism.  It is clear, therefore, that this action must correspond to the effective, isometric $\T^2_{uv}$ action on the biquotient $\cp^2 \# \cp^2 = \sph^3 \x \sph^3 \bq \T^2_{wz}$, for $\T^2_{wz}$ acting via (\ref{Eq:unifreeT2}).

From diagram (\ref{fig:dim4isotropy}) in Lemma \ref{lem:T2isotropy} it follows that the weighted orbit space of a $\T^2_{uv}$ action on $\M^4_{r,\lambda}$ (that is, the diagram of slopes of $\sph^1$ isotropy groups in $\T^2_{uv}$) is given by
\beq
\label{fig:dim4weights}
\xymatrix{
[\sph^1 \x \sph^1] \ar@{.}[rr]^{(1,0)} \ar@{.}[dd]_{(0,1)} & &
[\sph^1 \x \sph^1_j] \ar@{.}[dd]^{(2r+\lambda, 1)} \\ & & \\
[\sph^1_j \x \sph^1] \ar@{.}[rr]_{(1,0)} & &
[\sph^1_j \x \sph^1_j]
}
\eeq
where $r \in \Z$ and $\lambda \in \{0,1\}$.  When $\lambda = 0$ (resp. $\lambda = 1$) it is clear from the discussion in Section \ref{S:dim4} that all possible weighted orbit spaces for $\sph^2 \x \sph^2$ (resp. $\cp^2 \# - \cp^2$) are achieved.
\end{proof}

\begin{rem}
A simple observation is that for each pair $(r,s) \in \Z^2$ there is a three-dimensional orbifold $X^3_{r,s}$ such that $\sph^2 \x \sph^2$ and $\cp^2 \# - \cp^2$ arise as the total spaces of particular $\sph^1$-bundles over $X^3_{r,s}$.  Indeed, the three-torus action on $(\sph^3 \x \sph^3, \met_0)$ defined by
$$
(u,w,z) \star \bpm q_1 \\ q_2 \epm =
\bpm z \bar w^r \bar u^s \alpha_1 + z w^{r+1} u^s \beta_1 j \\
     w \bar u \alpha_2 + w u \beta_2 j \epm
$$
is effective and isometric.  Denote the quotient space $(\sph^3 \x \sph^3, \met_0) \bq T^3_{uwz}$ by $X^3_{r,s}$.  It follows from the work of Perelman \cite{P1,P2} that the orbifold $X^3_{r,s}$ is homeomorphic to $\sph^3$ as a topological manifold.  The subtori given by restriction to the $(z,w)$ and $(z,u)$ coordinates are both of the form (\ref{Eq:inffreeT2}) and have quotients $\cp^2 \# - \cp^2$ and $\sph^2 \x \sph^2$ respectively.  In each case the remaining effective and isometric ($u$ or $w$) circle action on the $4$-manifold has isotropy subgroups arranged according to the diagram
$$
\xymatrix{
[\sph^1 \x \sph^1] \ar@{.}[rr]^{\Z_2}_{[\sph^1 \x \sph^3]} \ar@{.}[dd]_{\Z_{|2(r+s)+1|}}^{[\sph^3 \x \sph^1]} & &
[\sph^1 \x \sph^1_j] \ar@{.}[dd]^{\Z_{|2(r-s)+1|}}_{[\sph^3 \x \sph^1_j]} \\ & & \\
[\sph^1_j \x \sph^1] \ar@{.}[rr]_{\Z_2}^{[\sph^1_j \x \sph^3]} & &
[\sph^1_j \x \sph^1_j]
}
$$
where the vertices are fixed points of the action.  In the context of Fintushel's classification of smooth, effective circle actions on smooth, compact, simply-connected $4$-manifolds via Seifert invariants $(\alpha_i, \beta_i)$ \cite{F1}, this shows that there are infinitely many smooth, effective circle actions on $\cp^2 \# - \cp^2$ and $\sph^2 \x \sph^2$ having the same $\alpha_i$ invariants (i.e. the orders of the isotropy groups), and so these actions must be distinguished by their (harder to compute) $\beta_i$ invariants.
\end{rem}

\begin{rem}
In the cases of $\cp^2 \# \pm \cp^2$ and $\sph^2 \x \sph^2$, Corollary \ref{Cor B} can also be proven directly, without appealing to Theorem A, by determining (up to equivariant diffeomorphism) all possible isometric circle actions on these manifolds in terms of Fintushel's classification of smooth, effective circle actions on smooth, compact, simply connected $4$-manifolds \cite{F1} and then performing computations similar to those in this section to show that each of these actions is realised by an effective, isometric circle action on the corresponding normal biquotient.  

For fixed point homogeneous circle actions, that is, those where the fixed point set of the action has a two-dimensional component (cf. \cite{GS} and \cite{GG}), all possible weighted orbit spaces (hence all such actions on these manifolds) were determined in \cite{GG}.  One can then reach the conclusion of Corollary \ref{Cor B} without using the work of Grove and Wilking \cite{GW}.  To determine the weighted orbit spaces for circle actions which are not fixed point homogeneous, one must combine Fintushel's classification with Theorem~2.4 in \cite{GW}.
\end{rem}

\section{Torus actions on five-dimensional biquotients}
\label{S:5dim}

From Oh's classification of smooth cohomogeneity-two torus actions on smooth $5$-manifolds \cite{Oh83}, we know that there is, up to equivariant diffeomorphism, a unique smooth $\T^3$ action on $\sph^5$.  This is, of course, realized by the linear (isometric) $\T^3$ action on the homogeneous space $\sph^5 = \SO(6)/\SO(5)$.  As we are interested only in those simply connected $5$-manifolds which admit a $\T^3$-invariant metric with non-negative curvature, that is, $\sph^5$, $\sph^3 \x \sph^2$ and $\sph^3 \tilde\x \sph^2$ (cf. \cite{GGS10}), we may restrict our attention from now on to smooth $\T^3$ actions on the manifolds $\sph^3 \x \sph^2$ and $\sph^3 \tilde\x \sph^2$.

Let $\sph^1$ act freely and isometrically on $(\sph^3 \x \sph^3, \met_0)$ via (\ref{Eq:freeS1}).  Denote the quotient $(\sph^3 \x \sph^3, \met_0) \bq \sph^1$ by $\M^5_{a,b,c,d}$.  By the discussion in Section \ref{S:Biquotients}, $\M^5_{a,b,c,d}$ is diffeomorpic to $\sph^3 \x \sph^2$ or $\sph^3 \tilde\x \sph^2$ depending on the parity of $a + b + c + d$.

Consider the effective, isometric $\T^4$ action on $(\sph^3 \x \sph^3, \met_0)$ in (\ref{Eq:effT4}).  The complementary three-dimensional torus $\T^3_{uvw}$ to the $z$-circle in $\T^4$ then acts on $\M^5_{a,b,c,d}$ effectively and isometrically via
\beq
\label{Eq:effT3}
(u,v,w) \star \bbm q_1 \\ q_2 \ebm =
 \bbm \bar w^n \alpha_1 + u w^k \beta_1 j \\
      w^m \alpha_2 + v w^\ell \beta_2 j \ebm
\eeq
where $(u,v,w) \in \T^3_{uvw}$, $\left[\bsm q_1 \\ q_2 \esm \right] \in M^5_{a,b,c,d}$, $k,\ell, m, n \in \Z$ and $am + cn = 1$.

\begin{lem}
\label{lem:T3isotropy}
Let $\T^3_{uvw}$ act effectively and isometrically on $M^5_{a,b,c,d}$ via the action (\ref{Eq:effT3}).  Then the four points $[\sph^1 \x \sph^1]$, $[\sph^1_j \x \sph^1]$, $[\sph^1 \x \sph^1_j]$, $[\sph^1_j \x \sph^1_j] \in M^5_{a,b,c,d}$ each have $\T^2$ isotropy, while a point has $\sph^1$ isotropy if and only if it lies in one of $[\sph^1 \x \sph^3]$, $[\sph^1_j \x \sph^3]$, $[\sph^3 \x \sph^1]$ or $[\sph^3 \x \sph^1_j]$.  In particular, the $\sph^1$ isotropy subgroups of $\T^3_{uvw}$ are arranged according to the diagram:

\beq
\label{fig:dim5isotropy}
\xymatrix{
[\sph^1 \x \sph^1] \ar@{.}[rrr]^{\{(1,v,1)\}}_{[\sph^1 \x \sph^3]} \ar@{.}[dd]_{\{(u,1,1)\}}^{[\sph^3 \x \sph^1]} & & &
[\sph^1 \x \sph^1_j] \ar@{.}[dd]^{\{(\bar w^{bn + ak}, \bar w^{dn + c\ell}, w^a)\}}_{[\sph^3 \x \sph^1_j]} \\ & & & \\
[\sph^1_j \x \sph^1] \ar@{.}[rrr]_{\{(w^{bm - ck}, w^{dm - c\ell}, w^c)\}}^{[\sph^1_j \x \sph^3]} & & &
[\sph^1_j \x \sph^1_j]
}
\eeq
\end{lem}

\begin{proof}
Suppose that $\left[\bsm q_1 \\ q_2 \esm \right]$ is fixed by some element $(u,v,w)$ of $\T^3_{uvw}$.  That is, there exists $z \in \sph^1$ such that
$$
\bpm \alpha_1 + \beta_1 j \\
     \alpha_2 + \beta_2 j \epm =
\bpm z^a \bar w^n \alpha_1 + u z^b w^k \beta_1 j \\
     z^c w^m \alpha_2 + v z^d w^\ell \beta_2 j \epm.
$$
Then each of the following four conditions must hold:
\begin{itemize}
\item $\alpha_1 = 0$ or $z^a \bar w^n = 1$;
\item $\beta_1 = 0$ or $u z^b w^k = 1$;
\item $\alpha_2 = 0$ or $z^c w^m = 1$;
\item $\beta_2 = 0$ or $v z^d w^\ell = 1$.
\end{itemize}
By effectiveness of the $\T^3_{uvw}$ action, it is clear that whenever all of $\alpha_1$, $\alpha_2$, $\beta_1$ and $\beta_2$ are non-zero, the isotropy group is trivial.

Suppose first that $\beta_1 = \beta_2 = 0$. Then $z^a \bar w^n = 1$ and $z^c w^m = 1$, from which it follows that $w = z = 1$, since $am + cn = 1$.  Thus the isotropy subgroup of the point $[\sph^1 \x \sph^1]$ is given by
$$
\T^2 \cong \{(u, v, 1)\} \subset \T^3_{uvw}.
$$
If $\alpha_1 = \alpha_2 = 0$, then $u = \bar z^b \bar w^k$ and $v = \bar z^d \bar w^\ell$, since $|\beta_1| = |\beta_2| = 1$.  Therefore the isotropy subgroup of the point $[\sph^1_j \x \sph^1_j]$ is
$$
\T^2 \cong \{(\bar z^b \bar w^k, \bar z^d \bar w^\ell, w)\} \subset \T^3_{uvw}.
$$
Suppose now that $\beta_1 = \alpha_2 = 0$.  Then $z^a \bar w^n = 1$ and $v = \bar z^d \bar w^\ell$.  Let $\zeta \in \sph^1$ such that $\zeta^a = w$.  It follows by setting $z = \zeta^n$ that $z^a = w^n$ and $v = \bar \zeta^{dn + a\ell}$.  Hence the isotropy subgroup of the point $[\sph^1 \x \sph^1_j]$ is given by
$$
\T^2 \cong \{(u, \bar \zeta^{dn + a\ell}, \zeta^a)\} \subset \T^3_{uvw}.
$$
We remark that $\gcd(a, dn+a\ell) = 1$, since $\gcd(ab,cd) = 1$ and $am + cn = 1$.

The analogous computation for $\alpha_1 = \beta_2 = 0$ shows that the isotropy subgroup of the point $[\sph^1_j \x \sph^1]$ is given by
$$
\T^2 \cong \{(\zeta^{bm - ck}, v, \zeta^c)\} \subset \T^3_{uvw}.
$$

Suppose, on the other hand, that $\beta_1 = 0$ and $\alpha_2, \beta_2 \neq 0$.  Then $z^a \bar w^n = 1$, $z^c w^m = 1$ and $v = \bar z^d \bar w^\ell$.  As before it follows from $am + cn = 1$ that $w = z = 1$, hence $v = 1$.  Therefore the isotropy subgroup of a generic point in $[\sph^1 \x \sph^3]$ is
$$
\sph^1 \cong \{(u,1,1)\} \subset \T^3_{uvw}.
$$
Similarly, generic points in $[\sph^3 \x \sph^1]$ (i.e. points for which $\beta_2 = 0$ and $\alpha_1, \beta_1 \neq 0$) have isotropy subgroup
$$
\sph^1 \cong \{(1,v,1)\} \subset \T^3_{uvw}.
$$
If now $\alpha_2 = 0$ and $\alpha_1, \beta_1 \neq 0$, then $v = \bar z^d \bar w^\ell$, $u = \bar z^b \bar w^k$ and $z^a \bar w^n = 1$.  As above, let $\zeta \in \sph^1$ such that $\zeta^a = w$ and set $z = \zeta^n$.  Then $u = \bar \zeta^{bn + ak}$ and $v = \bar \zeta^{dn + a\ell}$.  Therefore a generic point in $[\sph^3 \x \sph^1_j]$ has isotropy subgroup
$$
\sph^1 \cong \{(\bar \zeta^{bn + ak}, \bar \zeta^{dn + a\ell}, \zeta^a)\} \subset \T^3_{uvw}.
$$
The computation to show that a generic point in $[\sph^1_j \x \sph^3]$ (i.e. $\alpha_1 = 0$ and $\alpha_2, \beta_2 \neq 0$) has isotropy subgroup
$$
\sph^1 \cong \{(\zeta^{bm - ck}, \zeta^{dm - c\ell}, \zeta^c)\} \subset \T^3_{uvw}
$$
is completely analogous.
\end{proof}

\begin{prop}
\label{prop:5mnfds}
Every smooth, effective $\T^3$ action on either $\sph^3 \x \sph^2$ or $\sph^3 \tilde\x \sph^2$ is equivariantly diffeomorphic to an effective, isometric $\T^3_{uvw}$ action on the corresponding biquotient $M^5_{a,b,c,d} = (\sph^3 \x \sph^3, \met_0) \bq \sph^1$.
\end{prop}

\begin{proof}
From diagram \ref{fig:dim5isotropy} in Lemma \ref{lem:T3isotropy} it follows that the weighted orbit space of an effective, isometric $\T^3$ action on $M^5_{a,b,c,d}$ (that is, the diagram of slopes of $\sph^1$ isotropy groups in $\T^3_{uvw}$) is given by
\beq
\label{fig:dim5weights}
\xymatrix{
[\sph^1 \x \sph^1] \ar@{.}[rrr]^{(1,0,0)} \ar@{.}[dd]_{(0,1,0)} & & &
[\sph^1 \x \sph^1_j] \ar@{.}[dd]^{(- bn - ak, - dn - a\ell, a)} \\ & & & \\
[\sph^1_j \x \sph^1] \ar@{.}[rrr]_{(bm - ck, dm - c\ell, c)} & & &
[\sph^1_j \x \sph^1_j]
}
\eeq
where $\gcd(ab, cd) = 1$, $am + cn = 1$, and $k, \ell \in \Z$.  It is then clear from the discussion in Section \ref{S:dim5} that all possible (legally) weighted orbit spaces for $\sph^3 \x \sph^2$ and $\sph^3 \tilde\x \sph^2$ are achieved.
\end{proof}


\section{Principal circle bundles}
\label{S:P_bundles}

Let $\B$ be an arbitrary simply connected manifold.  It is well known that oriented (hence principal) circle bundles over $\B$ are classified by their Euler classes in the second integral cohomology $H^2 (\B; \Z)$.  Since $\B$ is simply connected, the total space $\p$ of the bundle is simply connected if and only if the Euler characteristic is primitive, i.e. a generator of $H^2 (\B; \Z)$.

As a consequence of their analysis of oriented circle bundles over compact, simply connected $4$-manifolds, Duan and Liang \cite{DL} have shown that if $\p$ is simply connected and the total space of a principal circle bundle over $\sph^2 \times \sph^2$ or $\cp^2 \# \pm \cp^2$, then it must be diffeomorphic to one of $\sph^3 \times \sph^2$ or $\sph^3 \tilde{\x} \sph^2$.  On the other hand, Grove and Ziller \cite[Theorem 4.5]{GZ} have observed that simply connected principal circle bundles over $\sph^2 \times \sph^2$ and $\cp^2 \# \pm \cp^2$ all arise by considering circle sub-actions of free (isometric) torus actions on $\sph^3 \x \sph^3$, namely
$$
\sph^1 = \T^2/ \sph^1 \lra (\sph^3 \x \sph^3) \bq \sph^1 \lra (\sph^3 \x \sph^3) \bq \T^2.
$$

From the discussions in previous sections, one can make the following observation.

\begin{prop}
\label{prop:PB}
Suppose $\T^2$ acts freely and isometrically on $\sph^3 \x \sph^3$.  Embed $\sph^1$ into $\T^2$ via $z \mapsto (z^p, z^q)$, $p,q \in \Z$, $\gcd(p,q)= 1$, and denote this circle subgroup $\sph^1_{p,q}$.  Consider the principal $\sph^1$-bundle
$$
\sph^1 = \T^2/\sph^1_{p,q} \lra X^5_{p,q} := (\sph^3 \x \sph^3) \bq \sph^1_{p,q} \lra \M^4 := (\sph^3 \x \sph^3) \bq \T^2.
$$
\begin{itemize}
\item[(a)]If $\M^4 = \sph^2 \x \sph^2$, then $X^5_{p,q} = \sph^3 \x \sph^2$ for all $p,q$.

\item[(b)] If $\M^4 = \cp^2 \# -\cp^2$, then $X^5_{p,q} = \sph^3 \x \sph^2$ for $p$ even and $X^5_{p,q} = \sph^3 \tilde\x \sph^2$ for $p$ odd.

\item[(c)] If $\M^4 =\cp^2 \# \cp^2$, then $X^5_{p,q} = \sph^3 \x \sph^2$ for $p+q$ even and $X^5_{p,q} = \sph^3 \tilde\x \sph^2$ for $p+q$ odd.

\end{itemize}
\end{prop}

\begin{proof}
Recall that every free, isometric $\T^2$ action on $\sph^3 \x \sph^3$ is equivalent to one of the $\T^2_{wz}$ actions (\ref{Eq:inffreeT2}) or (\ref{Eq:unifreeT2}) described in Section \ref{S:dim4}.  First consider $\sph^1_{p,q}$ as a sub-action of (\ref{Eq:inffreeT2}).  Then $\sph^1_{p,q}$ acts on $\sph^3 \x \sph^3$ via
$$
z \star \bpm q_1 \\ q_2 \epm =
\bpm z^{q-pr} \alpha_1 + z^{q + p(r + \lambda)} \beta_1 j \\
     z^p \alpha_2 + z^p \beta_2 j \epm.
$$
From (\ref{Eq:freeS1}) it follows that $a = q - pr$, $b = q + p(r + \lambda)$, $c = p$ and $d = p$.  Since $w_2(X^5_{p,q}) = a + b + c + d \in \Z_2$, in the cases of $\M^4 = \sph^2 \x \sph^2$ and $\M^4 = \cp^2 \# -\cp^2$ we are done.  The case $\M^4 = \cp^2 \# \cp^2$ is similar.
\end{proof}

It is now natural to ask whether every free (isometric) $\sph^1$ action on $\sph^3 \x \sph^3$ extends to a free (isometric) $\T^2$ action, and hence defines a principal circle bundle as above.  This amounts to asking whether the effective $\T^3$ action in (\ref{Eq:effT3}) contains a free circle sub-action.

\begin{lem}
\label{lem:extend}
Let $\M^5_{a,b,c,d} = (\sph^3 \x \sph^3, \met_0) \bq \sph^1$ be as in Section \ref{S:Biquotients} and let $\T^3_{uvw}$ act on $\M^5_{a,b,c,d}$ via the action (\ref{Eq:effT3}).  Up to reparametrization, a circle embedded in $\T^3_{uvw}$ via $u \mapsto (u^p,u^q,u^r)$, $\gcd(p,q,r) = 1$, acts freely on $\M^5_{a,b,c,d}$ if and only if $r = 1$ and
$$
a(q + \ell) + dn = \pm 1, \quad c(p + k) - bm = \pm 1, \ \ {\rm and} \ \ b(q + \ell) - d(p + k) = \pm 1
$$
for some choice of signs.  Furthermore, a necessary condition for such an $\sph^1$ action to be free is
\beq
\label{Eq:S1act}
bd \pm ac \pm ad \pm bc = 0
\eeq
for some choice of signs.
\end{lem}

\begin{proof}
The $\sph^1$ action on $M^5_{a,b,c,d}$ is given by
$$
u \star \bbm q_1 \\ q_2 \ebm =
\bbm \bar u^{rn} \alpha_1 + u^{p + rk} \beta_1 j  \\
      u^{rm} \alpha_2 + u^{q + r\ell} \beta_2 j \ebm
$$
We may assume that each of $q_n \in \sph^3$, $n = 1,2$, is in one of $\sph^1$ or $\sph^1_j$, since allowing both $\alpha_n$ and $\beta_n$ to be non-zero will only increase the number of freeness conditions to be satisfied.  Then $u \star \left[\bsm q_1 \\ q_2 \esm \right] = \left[\bsm q_1 \\ q_2 \esm \right]$ if and only if there exists some $z \in \sph^1$ such that
$$
\bpm \alpha_1 + \beta_1 j \\ \alpha_2 + \beta_2 j \epm =
\bpm \bar u^{rn} z^a \alpha_1 + u^{p + rk} z^b \beta_1 j  \\
      u^{rm} z^c \alpha_2 + u^{q + r\ell} z^d \beta_2 j \epm.
$$
That is, if and only if there is some $z \in \sph^1$ such that
\begin{enumerate}
\item \label{cond1} $\bar u^{rn} z^a  = 1$ and $u^{rm} z^c = 1$; or
\item \label{cond2} $\bar u^{rn} z^a  = 1$ and $u^{q + r\ell} z^d = 1$; or
\item \label{cond3} $u^{p + rk} z^b = 1$ and $u^{rm} z^c = 1$; or
\item \label{cond4} $u^{p + rk} z^b = 1$ and $u^{q + r\ell} z^d = 1$.
\end{enumerate}
In each case, if $u = 1$ then $z = 1$ by freeness of the $z$ action.  By \ref{cond1} we have $\bar u^{rcn} = \bar z^{ac} = u^{ram}$, from which it follows that $1 = u^{r(am + cn)} = u^r$.  Hence a necessary condition for freeness is that $r = \pm 1$.  Without loss of generality, assume that $r = 1$.

Similarly, \ref{cond2} implies that $\bar u^{dn} = \bar z^{ad} = u^{a(q + \ell)}$, hence $u^{dn + a(q + \ell)} = 1$.  Then $u = 1$ if and only if $dn + a(q + \ell) = \pm 1$.  The conditions \ref{cond3} and \ref{cond4} yield $c(p + k) - bm = \pm 1$ and $b(q + \ell) - d(p + k) = \pm 1$ in the same way.

Finally, equation (\ref{Eq:S1act}) arises by noticing that (for an appropriate choice of signs)
\begin{align*}
bd \pm ac \pm ad \pm bc &= bd(am + cn) + ac(b(q + \ell) - d(p + k)) \\
         &\qquad+ ad(c(p + k) - bm) - bc(dn + a(q + \ell))\\
         &= 0.
\end{align*}
\end{proof}

\begin{cor}
There exist infinitely many free (isometric) $\sph^1$ actions on $\sph^3 \x \sph^3$ which do not extend to a free (isometric) $\T^2$ action, and infinitely many descriptions of $\sph^3 \x \sph^2$ and $\sph^3 \tilde\x \sph^2$ as biquotients which do not admit a free (isometric) $\sph^1$ action.
\end{cor}

\begin{proof}
Let $\sph^1$ act on $\sph^3 \x \sph^3$ via (\ref{Eq:freeS1}) with $a = -1$, $b = 3$, $c = 15k + 1$ and $d = 5$, $k \in \Z$.  The action is free, since $\gcd(ab, cd) = 1$.  Since the second Stiefel-Whitney class of the quotient $\M^5_{a,b,c,d}$ is $w_2(\M^5_{a,b,c,d}) = a + b + c + d = k \in \Z_2$, it therefore follows that both $\sph^3 \x \sph^2$ and $\sph^3 \tilde\x \sph^2$ arise, depending on the parity of $k$.  Finally, we note that equation (\ref{Eq:S1act}) can never be satisfied for any value of $k$.
\end{proof}


\end{document}